\documentclass[11pt,a4paper]{amsart}

\usepackage{amssymb,amsmath}
\usepackage{amsthm}
\usepackage{latexsym}
\usepackage{hyperref}
\usepackage{graphics}
\usepackage{euscript}
\usepackage[dvips]{graphicx}
\usepackage{mathrsfs}
\usepackage{mathabx}

\hypersetup{
 colorlinks,
 citecolor=blue,
 filecolor=black,
 linkcolor=blue,
 urlcolor=black
}

\newtheorem{thm}{Theorem}[section]
\newtheorem{prop}[thm]{Proposition}
\newtheorem{cor}[thm]{Corollary}
\newtheorem{lem}[thm]{Lemma}

\theoremstyle{definition}

\theoremstyle{remark}
\newtheorem{example}[thm]{Example}

{\em}

{\em}

\newcounter{substep}
\def\thesubstep{\arabic{substep}}

\newenvironment{substeps}[1]{%
\refstepcounter{substep}\noindent{(\ref{#1}.\thesubstep)\ }\ }%
{\em}

\newcounter{subsubstep}
\def\thesubsubstep{\arabic{subsubstep}}

{\em}

\newcommand{\K}{{\mathbb K}} 
 \newcommand{\R}{{\mathbb R}}
 \newcommand{\C}{{\mathbb C}}
 \newcommand{\F}{{\mathbb F}}
 
\newcommand{\sph}{{\mathbb S}} 
 \newcommand{\PP}{{\mathbb P}}

\newcommand{\cl}{\operatorname{Cl}}

\newcommand{\im}{\operatorname{im}}

\newcommand{\zar}{\operatorname{zar}}

\newcommand{\tr}{\operatorname{tr}}

\newcommand{\x}{{\tt x}} \newcommand{\y}{{\tt y}}
 \renewcommand{\t}{{\tt t}}

\newcommand{\ol}{\overline}

\title[On the one dimensional polynomial and regular images of $\R^n$]{On the one dimensional polynomial\\ and regular images of $\R^n$}

\date{}

\author{Jos\'e F. Fernando}
\address{Departamento de \'Algebra, Facultad de Ciencias Matem\'aticas, Universidad Complutense de Madrid, 28040 MADRID (SPAIN)}
\email{josefer@mat.ucm.es}
\thanks{Author supported by Spanish GR MTM2011-22435. 
}

\begin{document}

\begin{abstract}
In this work we present a full geometric characterization of the $1$-dimensional polynomial and regular images $S$ of $\R^n$ and we compute for all of them the invariants ${\rm p(S)}$ and ${\rm r}(S)$, already introduced in \cite{fg2}.
\end{abstract}
\date{08/11/2012}

\maketitle

\section{Introduction}\label{s1}
A map $f:=(f_1,\ldots,f_m):\R^n\to\R^m$ is a {\em polynomial map} if each of its components $f_i\in\R[\x]:=\R[\x_1,\ldots,\x_n]$. A subset $S$ of $\R^m$ is a {\em polynomial image} of $\R^n$ if there exists a polynomial map $f:\R^n\to\R^m$ such that $S=f(\R^n)$. Let $S$ be a subset of $\R^m$; we define
$$
{\rm p}(S):=
\left\{
\begin{array}{ll}
\text{least $p\geq 1$}&\text{ such that $S$ is a polynomial image of $\R^p$,}\\[6pt]
+\infty&\text{ otherwise.}
\end{array}
\right.
$$
More generally, a map $f:=(f_1,\ldots,f_m):\R^n\to\R^m$ is a {\em regular map} if each component $f_i$ is a regular function of $\R(\x):=\R(\x_1,\ldots,\x_n)$, that is, each $f_i=g_i/h_{i}$ is a quotient of polynomials such that the zero set of $h_{i}$ is empty. Analogously, a subset $S$ of $\R^m$ is a {\em regular image} of $\R^n$ if it is the image $S=f(\R^n)$ of $\R^n$ by a regular map $f$ and we define the invariant
$$
{\rm r}(S):=
\left\{
\begin{array}{ll}
\text{least $r\geq 1$}&\text{ such that $S$ is a regular image of $\R^r$,}\\[6pt]
+\infty&\text{ otherwise.}
\end{array}
\right.
$$
Obviously, ${\rm r}(S)\leq{\rm p}(S)$ and by Tarski's Theorem (see \cite[2.8.8]{bcr}) the dimension $\dim S$ of $S$ is less than or equal to both of them. Of course, the inequalities $\dim S\leq{\rm r}(S)\leq{\rm p}(S)$ can be strict and it may happen that the second invariant is finite while the third is infinite, even if $S\subset\R$ (see Lemma \ref{interval}).

A celebrated Theorem of Tarski-Seidenberg \cite[1.4]{bcr} says that the image of any polynomial map (and more generally of a regular map) $f:\R^m\rightarrow\R^n$ is a semialgebraic subset $S$ of $\R^n$, that is, it can be written as a finite boolean combination of polynomial equations and inequalities, which we will call a \em semialgebraic \em description. By \em elimination of quantifiers, \em $S$ is semialgebraic if it has a description by a first order formula \em possibly with quantifiers\em. Such a freedom gives easy semialgebraic descriptions for topological operations: interiors, closures, borders of semialgebraic sets are again semialgebraic.

In an \em Oberwolfach \em week \cite{g}, Gamboa proposed to characterize the semialgebraic sets of $\R^m$ which are polynomial images of $\R^n$ for some $n\geq1$. The interest of polynomial (and also regular) images is far from discussion since there are many problems in Real Algebraic Geometry that for such sets can be reduced to the case $S=\R^n$ (see \cite{fg1}, \cite{fg2} or \cite{fgu} for further comments). Examples of such problems are:
\begin{itemize}
\item optimization of polynomial (and/or regular) functions on $S$, 
\item characterization of the polynomial (or regular functions) which are positive semidefinite on $S$ (Hilbert's 17th problem and Positivestellensatz),
\end{itemize}

As we have already pointed out in \cite{fg1} there are some straightforward properties that a regular image $S\subset\R^m$ much satisfy: it must be pure dimensional, connected, semialgebraic and its Zariski closure must be irreducible. Furthermore, $S$ must be, by \cite[3.1]{fg3}, \em irreducible \em in the sense that its ring ${\mathcal N}(S)$ of Nash functions on $S$ is an integral domain. Recall here that a \em Nash function \em on an open semialgebraic subset $U\subset\R^m$ is an analytic function that satisfies a nontrivial polynomial equation, that is, there exists $P\in\R[\x,\y]$ such that $P(x,f(x))=0$ for all $x\in U$. Now, the ring ${\mathcal N}(S)$ of Nash functions on $S$ is the collection of all functions on $S$ that admit a Nash extension to an open semialgebraic neighborhood $U$ of $S$ in $\R^m$, endowed with the usual sum and product (for further details see \cite{fg3}). 

In this work we focus our attention in the one dimensional case and we present a full geometric characterization of the polynomial and regular one dimensional images of $\R^n$; in fact, we compute the exact value of the invariants ${\rm p}$ and ${\rm r}$ for all of them. As we will see along this work in the one dimensional case the only three possible values for both invariants ${\rm p}$ and ${\rm r}$ are $1,2$ or $+\infty$. In fact, all the possibilities with the restriction $1\leq{\rm r}\leq{\rm p}\leq +\infty$ are attained except for the pair ${\rm r}=1$ y ${\rm p}=2$ which is not attainable (see Theorems \ref{dim1p} and \ref{dim1r}, Propositions \ref{p12} and \ref{r12}, Corollary \ref{nor1p2} and Lemma \ref{interval} to complete the picture). We provide the following table illustrating the situation.

\vspace{2mm}
\begin{center}
{\setlength{\arrayrulewidth}{.5pt}
\renewcommand*{\arraystretch}{1.5}
\begin{tabular}{|c|c|c|c|c|c|c|}
\hline
$S$&$\R$ \'o $[0,+\infty)$&$-$&$[0,1)$&$(0,+\infty)$&$(0,1)$&Any non rational algebraic curve\\
\hline
${\rm r}(S)$&$1$&$1$&$1$&$2$&$2$&$+\infty$\\
\hline
${\rm p}(S)$&$1$&$2$&$+\infty$&$2$&$+\infty$&$+\infty$\\
\hline
\end{tabular}}
\end{center}

\vspace{2mm}
We recall that the study of the one dimensional polynomial images of $\R^n$ was partially and naively approached before in our previous work \cite[\S2]{fg2}, but without presenting any conclusive result. 

\subsection*{Notations and terminology}
Before stating our main results, whose proofs are developed in Section \ref{s3} after the preparatory work of Section \ref{s2}, we recall some preliminary standard notations and terminology. We write $\K$ to refer indistinctly to $\R$ or $\C$ and we denote by $\mathsf{H}_\infty(\K):=\{x_0=0\}$ the hyperplane of infinity of the projective space $\K\PP^m$, which contains $\K^m$ as the set $\K\PP^m\setminus\mathsf{H}_\infty(\K)=\{x_0=1\}$. In case $m=1$, we denote $\{p_\infty\}:=\{x_0=0\}$ the point of infinity of the projective line $\K\PP^1$.

For each $n\geq1$ denote by
$$
\sigma_n:\C\PP^n\to\C\PP^n,\ z=(z_0:z_1:\cdots:z_n)\mapsto\ol{z}=(\ol{z_0}:\ol{z_1}:\cdots:\ol{z_n})
$$
the complex conjugation. Clearly, $\R\PP^n$ is the set of fixed points of $\sigma_n$. A set $A\subset\C\PP^n$ is called \em invariant \em if $\sigma(A)=A$. It is well-know that if $Z\subset\C\PP^n$ is an invariant nonsingular (complex) projective variety, then $Z\cap\R\PP^n$ is a nonsingular (real) projective variety. We also say that a rational map $h:\C\PP^n\dashrightarrow\C\PP^m$ is \em invariant \em if $h\circ\sigma_n=\sigma_m\circ h$. Of course, $h$ is invariant if its components can be chosen as homogeneous polynomials with real' coefficients; hence, it provides by restriction a ``real'' rational map $h|_{\R\PP^n}:\R\PP^n\dashrightarrow\R\PP^m$.

Given a semialgebraic set $S\subset\R^m\subset\R\PP^m\subset\C\PP^m$, we denote by $\cl_{\K\PP^m}^{\zar}(S)$ its Zariski closure in $\K\PP^m$. Obviously, $\cl_{\C\PP^m}^{\zar}(S)\cap\R\PP^m=\cl_{\R\PP^m}^{\zar}(S)$ and $\cl^{\zar}(S)=\cl_{\R\PP^m}^{\zar}(S)\cap\R^m$ is the \em Zariski closure of $S$ in $\R^m$\em. We denote by $S_{(k)}$ the set of points of $S$ that have  local dimension $k$. 

Recall that a \em complex rational curve \em is the image of $\C\PP^1$ under a birational (and hence regular) map while a \em real rational curve \em is a real projective irreducible curve $C$ such that $C_{(1)}$ is the image of $\R\PP^1$ under a birational (and hence regular) map (see Lemma \ref{curve}).

\subsection*{\bf Main results.} We begin by a geometrical characterization of the $1$-dimensional polynomial images of Euclidean spaces (that is, those with ${\rm p}=1,2$, see also \cite[2.1-2]{fg2}) and after that we determine those with ${\rm p}=1$.

\begin{thm}\label{dim1p}
Let $S\subset\R^m$ be a $1$-dimensional semialgebraic set. Then, the following assertions are equivalent:
\begin{itemize}
\item[(i)] $S$ is a polynomial image of $\R^n$ for some $n\geq1$.
\item[(ii)] $S$ is irreducible, unbounded and $\cl_{\C\PP^m}^{\zar}(S)$ is an invariant rational curve such that $\cl_{\C\PP^m}^{\zar}(S)\cap\mathsf{H}_\infty(\C)$ is a singleton $\{p\}$ and the germ $\cl_{\C\PP^m}^{\zar}(S)_p$ is irreducible.
\end{itemize}
In particular, if that is the case ${\rm p}(S)\leq2$.
\end{thm}

\begin{prop}\label{p12}
Let $S\subset\R^m$ be a $1$-dimensional semialgebraic set which is a polynomial image of $\R^n$ for some $n\geq1$. Then, ${\rm p}(S)=1$ if and only if $S$ is closed in $\R^m$. 
\end{prop}

The counterpart of the previous results in the regular setting consist of the full geometric characterization of the $1$-dimensional regular images of Euclidean spaces and the description of those with ${\rm r}=1$.

\begin{thm}\label{dim1r}
Let $S\subset\R^m$ be a $1$-dimensional semialgebraic set. Then, the following assertions are equivalent:
\begin{itemize}
\item[(i)] $S$ is a regular image of $\R^n$ for some $n\geq1$.
\item[(ii)] $S$ is irreducible and $\cl_{\R\PP^m}^{\zar}(S)$ is a rational curve.
\end{itemize}
In particular, if that is the case ${\rm r}(S)\leq2$.
\end{thm}

\begin{prop}\label{r12}
Let $S\subset\R^m$ be a $1$-dimensional semialgebraic set which is a regular image of $\R^n$ for some $n\geq1$. Then, ${\rm r}(S)=1$ if and only if either 
\begin{itemize}
\item[(i)] $\cl_{\R\PP^m}(S)=S$ or 
\item[(ii)] $\cl_{\R\PP^m}(S)\setminus S=\{p\}$ is a singleton and the analytic closure of the germ $S_p$ is irreducible. 
\end{itemize}
\end{prop}

\begin{cor}\label{nor1p2}
There is no $1$-dimensional semialgebraic set $S\subset\R^m$ with ${\rm p}(S)=2$ and ${\rm r}(S)=1$.
\end{cor}
\begin{proof}
Suppose, by way of contradiction, that there exists a semialgebraic set $S\subset\R^m$ with $\dim S=1$, ${\rm p}(S)=2$ and ${\rm r}(S)=1$. By Theorem \ref{dim1p} and Proposition\ref{p12}, we deduce that $S$ is unbounded and $S$ is not closed in $\R^m$. Thus, $\cl_{\R\PP^m}(S)\setminus S$ has at least two elements: one point in $\mathsf{H}_\infty(\R)$ because $S$ is unbounded and another one in $\R^m$, since $S$ is not closed in $\R^m$. But by Proposition \ref{r12}, $\cl_{\R\PP^m}(S)\setminus S$ is either empty or a singleton, a contradiction; hence, the statement follows.
\end{proof}

\section{Main tools}\label{s2}

In this section, we present the main tools used to prove the results presented in this article. In what follows, we will use freely usual concepts of (complex) Algebraic Geometry as: rational map, regular map, normalization, etc. and we refer the reader to \cite{m,sh1} for further details. We recall first here the following useful and well-known fact concerning the regularity of rational maps defined on a nonsingular curve (see \cite[7.1]{m}) that will be used several times in the sequel.

\begin{lem}\label{curve}
Let $Z\subset\C\PP^n$ be a nonsingular projective curve and let $F:Z\dashrightarrow\C\PP^m$ be a rational map. Then, $F$ extends (uniquely) to a regular map $F':Z\to\C\PP^m$. Moreover, if $Z,F$ are invariant so is $F'$.
\end{lem}

\subsection*{Normalization of an algebraic curve} A main tool will be the normalization $(\widetilde{X},\Pi)$ of an either affine or projective algebraic curve $X$, both in the real and in the complex case. As it is well-know, the normalization is birationally equivalent to $X$ and so unique up to biregular homeomorphism; furthermore, if $X$ is an invariant complex algebraic curve, we may assume that also both $\widetilde{X}$ and $\pi$ are invariant. To prove this, one can construct $(\widetilde{X},\pi)$ as the desingularization of $X$ via a finite chain of suitable invariant blowing-ups. Recall also that all the fibers of $\Pi:\widetilde{X}\to X$ are finite and if $x\in X$ is a non singular point, then the fiber of $x$ is a singleton. Moreover, if $X$ is complex, then the cardinal of the fiber of a point $x\in X$ coincides with the number of irreducible components of the germ $X_x$. If $X\subset\R^m$ is an affine algebraic curve, $Y:=\cl_{\C\PP^m}^{\zar}(X)$ and $(\widetilde{Y}\subset\C\PP^k,\Pi)$ is an invariant normalization of $Y$, we have that: 
\begin{itemize}
\item $(\widetilde{Z}:=\widetilde{Y}\cap\R\PP^k,\Pi|_{\widetilde{Z}})$ is the normalization of $Z:=\cl_{\R\PP^m}^{\zar}(X)$ and $\Pi(\widetilde{Z})=Z_{(1)}$,
\item $(\widetilde{X}:=\widetilde{Y}\cap\R^k,\pi:=\Pi|_{\widetilde{X}})$ is the normalization of $X$ and $\pi(\widetilde{X})=X_{(1)}$.\qed
\end{itemize}

\vspace{2mm}
Next two results are the clue to prove the Main results stated in the Introduction.

\begin{lem}\label{fact1}
Let $f:\R\to\R^m$ be a non-constant rational map and let $S:=f(\R)$. Then, 
\begin{itemize}
\item[(i)] $f$ extends (uniquely) to an invariant regular map $F:\C\PP^1\to\C\PP^m$ such that $F(\C\PP^1)=\cl_{\C\PP^m}^{\zar}(S)$.
\item[(ii)] $\cl_{\C\PP^m}^{\zar}(S)$ is an invariant rational curve and if $(\C\PP^1,\Pi)$ is an invariant normalization of $\cl_{\C\PP^m}^{\zar}(S)$ there is an invariant surjective regular map $\widetilde{F}:\C\PP^1\to\C\PP^1$ such that $F=\Pi\circ\widetilde{F}$.
\item[(iii)] If $f$ is polynomial, then we may choose $\Pi$ and $\widetilde{F}$ so that $\pi:=\Pi|_{\R}$ and $\widetilde{f}:=\widetilde{F}|_{\R}$ are polynomial. In particular, $\cl_{\C\PP^m}^{\zar}(S)\cap\mathsf{H}_\infty(\C)$ is a singleton $p$ and the germ $\cl_{\C\PP^m}^{\zar}(S)_p$ is irreducible.
\end{itemize}
\end{lem}
\begin{proof}
(i) Observe that $f$ extends naturally to an invariant rational map 
$$
F:=(F_0:F_1:\cdots:F_m):\C\PP^1\dashrightarrow\C\PP^m,
$$
where $F_i\in\R[\x_0,\x_1]$ are homogeneous polynomials of the same degree $d$. In fact, such extension is, by Lemma \ref{curve}, regular and unique. Moreover, since $S=f(\R)$, then $F(\C\PP^1)\subset\cl_{\C\PP^m}^{\zar}(S)$ contains, by \cite[2.31]{m}, a non-empty Zariski open subset of the $\cl_{\C\PP^m}^{\zar}(S)$; hence, since $F$ is proper and $\cl_{\C\PP^m}^{\zar}(S)$ is irreducible, we conclude, by \cite[2.33]{m}, that $F(\C\PP^1)=\cl_{\C\PP^m}^{\zar}(S)$.

(ii) Let $(\widetilde{Y}\subset\C\PP^k,\Pi)$ be a $\sigma$-invariant normalization of $Y:=\cl_{\C\PP^m}^{\zar}(S)$. Now, the composition $\Pi^{-1}\circ F:\C\PP^1\dashrightarrow \widetilde{Y}$ defines an invariant rational map that extends to an invariant surjective regular map $\widetilde{F}:\C\PP^1\to\widetilde{Y}$ such that $F=\Pi\circ\widetilde{F}$. Now, by \cite[7.6, 7.20]{m}, $\widetilde{Y}$ is a smooth curve of arithmetic genus $0$, that is, a smooth rational curve (see \cite[7.17]{m}); hence, we may take $\widetilde{Y}=\C\PP^1$. Thus, $(\R\PP^1,\Pi|_{\R\PP^1})$ is the normalization of $\cl_{\R\PP^m}^{\zar}(S)$ and $\Pi(\R\PP^1)=\cl_{\R\PP^m}^{\zar}(S)_{(1)}$.

(iii) If $f$ is polynomial, then $F_0:=\x_0^d$. Write $\Pi:=(\Pi_0,\ldots,\Pi_m)$ and $\widetilde{F}:=(\widetilde{F}_0,\widetilde{F}_1)$, where $\Pi_i,\widetilde{F}_j\in\R[\x_0,\x_1]$ are homogeneous polynomials, and let us check that we may assume $\Pi_0=\lambda\x_0^e$ and $\widetilde{F}_0=\mu\x_0^\ell$ for some positive integers $e,\ell$ such that $d=e\ell$; hence, $\pi:=\Pi|_{\R}$ and $\widetilde{f}:=\widetilde{F}|_{\R}$ are polynomial. 

Indeed, observe first that $\widetilde{F}$ is non constant because it is surjective. Factorize
$$
\Pi_0=\prod_{i=1}^e(a_i\x_1-b_i\x_0)\in\C[\x_0,\x_1]
$$
where $a_i,b_i\in\C$ and $(a_i,b_i)\neq(0,0)$ for $i=1,\ldots,m$. Let us check that all the factors $a_i\x_1-b_i\x_0$ are proportional. Denote $p_i:=\widetilde{F}_i(1,\x_1)$ and observe that
$$
\prod_{i=1}^e(a_ip_1-b_ip_0)=\Pi_0(p_0,p_1)=F_0(1,\x_1)=1;
$$
hence, all the factors in the previous expression are non zero constants $c_i\in\C$. Suppose that two of the pairs $(a_i,b_i)$ are not proportional, for instance, $(a_1,b_1)$ and $(a_2,b_2)$ are not proportional. Then, $(p_0,p_1)$ is the unique solution of the linear system 
$$
\left\{
\begin{array}{l}
a_1\x_1-b_1\x_0=c_1,\\
a_1\x_2-b_2\x_0=c_2
\end{array}\right.
$$
and so $p_0,p_1\in\C$, which contradicts the fact that $\widetilde{F}$ is non constant. Thus, we may write $\Pi_0=\pm(a\x_1-b\x_0)^e$ where $a,b\in\R$ and $(a,b)\neq(0,0)$. Consider an invariant change of coordinates $\Psi:\C\PP^1\to\C\PP^1$ that transforms $(a:b)$ onto $(0:1)$ and define $\Pi':=\Pi\circ\Psi$. Of course, $(\C\PP^1,\Pi')$ is an invariant normalization of $\cl_{\C\PP^m}^{\zar}(S)$ with $\Pi'_0=\lambda\x_0^e$. Define $\widetilde{F}'$ as the regular extension of $(\Pi')^{-1}\circ F$ to $\C\PP^1$; in particular, $F=\Pi'\circ\widetilde{F}'$. Since $\lambda(\widetilde{F}'_0)^e=\x_0^d$, we conclude $\widetilde{F}_0'=\mu\x_0^\ell$.

Finally, we have that $\Pi^{-1}(\cl_{\C\PP^m}^{\zar}(S)\cap\mathsf{H}_\infty(\C))=\{(0:1)\}$ and so, we deduce that $\cl_{\C\PP^m}^{\zar}(S)\cap\mathsf{H}_\infty(\C)=\{p\}$ is a singleton and that the germ $\cl_{\C\PP^m}^{\zar}(S)_p$ is irreducible.
\end{proof}

\begin{lem}\label{fact2}
Let $f:=(f_1,\ldots,f_m):\R^n\to\R^m$ be a non-constant rational map such that its image $f(\R^n)$ has dimension $1$. Then,
\begin{itemize}
\item[(i)] $f$ factors through $\R$, that is, there exist a rational function $g\in\R(\x)$ and a rational map $h:\R\to\R^m$ such that $f=h\circ g$.
\item[(ii)] If $f$ is moreover a polynomial map, we may also take $g$ and $h$ polynomial.
\end{itemize}
\end{lem}
\begin{proof}
Let $\F:=\R(f_1,\ldots,f_m)$ be the smallest subfield of the field of rational functions $\R(\x)$ in $n$ variables that contains $\R$ and $f_1,\ldots,f_m$. Note that $\tr\!.\!\deg(\F|\R)=\dim(\im f)=1$ we may assume that $f_1\not\in\R$. Thus, by L\"uroth's Theorem, there exists a rational function $g\in\R(\x)\setminus\R$ such that $\F=\R(g)$. Since $f_i\in\F=\R(g)$, we have $f_i=\frac{P_i(g)}{Q_i(g)}$ for some coprime polynomials $P_i,Q_i\in\R[\t]$. Now, the rational map $h:=(\frac{P_1}{Q_1},\ldots,\frac{P_m}{Q_m}):\R\to\R^m$, satisfies $f=h\circ g$ and so (i) holds.

Next, suppose that $f$ is moreover polynomial. Following \cite[Lem.2]{sch2} (see also \cite[Lem. 2, pag. 710-711]{sch}):

\vspace{1mm}
\begin{substeps}{fact2}\label{sch}
\em We may assume that the L\"uroth's generator $g$ of $\F$ is in fact polynomial. 
\end{substeps}

\vspace{1mm}
Now, by Bezout's Lemma, we can write $1=P_iA_i+Q_iB_i$ for some $A_i,B_i\in\R[\t]$. Substituting the variable $\t$ by $g$ we get the polynomial identity
$$
1=P_i(g)A_i(g)+Q_i(g)B_i(g)=Q_i(g)f_iA_i(g)+Q_i(g)B_i(g)=Q_i(g)\big(f_iA_i(g)+B_i(g)\big);
$$
hence, $Q_i(g)$ is a nonzero constant, and so the polynomials $h_i:=\frac{P_i(\t)}{Q_i(g)}$ fit our situation.

For the sake of completeness let us include the elementary proof of \cite[Lem.2]{sch2} that shows statement \ref{fact2}.\ref{sch}. Let $g_0\in\R(\x)\setminus\R$ be a L\"uroth's generator of $\F$. Since the extension $\F|\R$ has transcendence degree $1$, we may assume that $F:=f_1\in\R[\x]\setminus\R$. Let $R,S\in\R[\t]\setminus\{0\}$ and $P,Q\in\R[\x]\setminus\{0\}$ be pairs of relatively prime polynomials such that 
$$
F=\frac{R(g_0)}{S(g_0)}\quad\&\quad g_0=\frac{P}{Q}\ \Longrightarrow\ F=\frac{Q^rR(P/Q)}{Q^sS(P/Q)}Q^{s-r},
$$
where $r:=\deg(R)$ and $s:=\deg(S)$. Notice that the polynomials $Q$, $Q^rR(P/Q)$ and $Q^sS(P/Q)$ are pairwise relatively prime; once this is show, it follows directly, using the fact that $\R[\x]$ is a UFD, that $\beta:=Q^sS(P/Q)\in\R$ and $s-r\geq0$. 

Indeed, it is straightforward to show, using that $R,S\in\R[\t]$, that 
$$
\gcd(Q,Q^rR(P/Q))=\gcd(Q,Q^sS(P/Q))=\gcd(P,Q)=1.
$$
Next, by Bezout's Lemma, we find polynomials $A_1,A_2\in\R[\t]$ of degrees $k_i:=\deg(A_i)$ such that $1=A_1R+A_2S$ and substituting $\t\leadsto P/Q$ and multiplying the expresion by $Q^\ell$ where $\ell:=\max\{\deg(A_1)+\deg(R),\deg(A_2)+\deg(S)\}$, we get
$$
Q^{\ell}=Q^{\ell-k_1-r}(Q^{k_1}A_1(P/Q))(Q^rR(P/Q))+Q^{\ell-k_2-r}(Q^{k_2}A_2(P/Q))(Q^rS(P/Q))
$$
and so the $\gcd(Q^rR(P/Q),Q^sS(P/Q))$ divides $Q^\ell$; hence, 
$$
\gcd(Q^rR(P/Q),Q^sS(P/Q))=\gcd(Q^rR(P/Q),Q^sS(P/Q),Q^\ell)=1.
$$
Factorize $S=\alpha(\t-\xi_1)\cdots(\t-\xi_s)$ where $\alpha\in\R\setminus\{0\}$ and $\xi_i\in\C$. We have
$$
\beta=Q^sS(P/Q)=\alpha(P-\xi_1Q)\cdots(P-\xi_sQ);
$$
whence, $(P-\xi_iQ)=\gamma_i\in\C$ for $1\leq i\leq s$. If any two $\xi_i's$ were distinct, for instance, $\xi_1\neq\xi_2$, we would get $(\xi_2-\xi_1)Q=\gamma_1-\gamma_2\in\C$; hence, $Q\in\R[\x]\cap\C=\R$ and $P\in\R[\x]\cap\C=\R$, which contradicts the fact that $g_0=P/Q\in\R(\x)\setminus\R$. Thus, $S=\alpha(\t-\xi)^s$ where $\alpha,\xi\in\R$ and $s\geq0$.

If $s=0$, we may assume $Q=1$ and $g:=g_0=P\in\R[\x]$. If $s>0$, then $P-\xi Q=\gamma\in\R$ and so $g_0=P/Q=\xi+\gamma/Q$; hence, $\F=\R(g_0)=\R(\xi+\gamma/Q)=\R(\gamma/Q)=\R(g)$, where $g:=Q\in\R[\x]$, as wanted. 
\end{proof}

We finish this section with an elementary crucial example.

\begin{example}\label{s1rp1}
$\sph^1$ and $\R\PP^1$ are regular images of $\R$. Since $\R\PP^1$ is the image of $\sph^1$ via the canonical projection $\pi:\sph^1\to\R\PP^1$, it is enough to prove that $\sph^1$ is a regular image of $\R$. To that end, we may take for instance the regular map
$$
f:\R\to\sph^1,\ t\mapsto\Big(\Big(\frac{t^2-1}{t^2+1}\Big)^2-\Big(\frac{2t}{t^2+1}\Big)^2,2\Big(\frac{t^2-1}{t^2+1}\Big)\Big(\frac{2t}{t^2+1}\Big)\Big).
$$
Observe that the previous map is the composition of the inverse of the stereographic projection of $\sph^1$ from $(1,0)$ with 
$$
g:\C\equiv\R^2\to\C\equiv\R^2,\ z=x+\sqrt{-1}y\equiv(x,y)\mapsto z^2\equiv(x^2-y^2,2xy).
$$
\end{example}

\section{Proofs of the main results}\label{s3}

The purpose of this section is to prove Theorems \ref{dim1p} and \ref{dim1r} and Propositions \ref{p12} and \ref{r12}. We begin approaching the case $m=1$, that is, $S:=I$ is an interval of $\R$.

\begin{lem}\label{interval}
Let $I\subset\R$ be an interval. Then, 
\begin{itemize}
\item[(i)] ${\rm p}(I)<+\infty$ if and only if $I$ is unbounded. Moreover, if such is the case ${\rm p}(I)\leq 2$ and ${\rm p}(I)=2$ if and only if $I\subsetneq\R$ is moreover open.
\item[(ii)] ${\rm r}(I)\leq 2$. Moreover, ${\rm r}(I)=2$ if and only if $I\subsetneq\R$ is open.
\end{itemize}
\end{lem}
\begin{proof}
(i) First, if $f:\R\to\R$ is a non-constant polynomial map, the image of $f$ is either $\R$ or a proper closed unbounded interval; hence, if $I\subsetneq\R$ is open, then ${\rm p}(I)\geq 2$. In any case, if ${\rm p}(I)=n<+\infty$ and $g:\R^n\to\R$ is a polynomial map such that $g(\R^n)=I$, we take $x_0\in\R^n$ with $g(x_0)\neq g(0)$ and consider the non-constant polynomial map $f:\R\to\R,\ t\mapsto g(tx_0)$; hence, $f(\R)\subset I$ is unbounded.

To finish, it is enough to prove that the interval $[0,+\infty)$ is a polynomial image of $\R$ while $(0,+\infty)$ is a polynomial image of $\R^2$. To that end, consider the polynomial maps
$$
f_1:\R\to\R,\ t\mapsto t^2\qquad\text{and}\qquad f_2:\R^2\to\R,\ (x,y)\mapsto(xy-1)^2+x^2.
$$

(ii) For the second part of this assertion observe that, by Lemma \ref{curve}, a regular map $f:\R\to\R$, extends regularly to a map $F:\R\PP^1\to\R\PP^1$. Thus, the image of $F$ is either $\R\PP^1$ or a proper closed interval $J$ of $\R\PP^1$. If $F(p_\infty)=p_\infty$, then $I=J\setminus\{p_\infty\}$ is an unbounded closed interval of $\R$. On the other hand, if $F(p_\infty)=c\in\R$, then $J=[a,b]$ is a bounded closed interval of $\R$ and $I$ is either equal to $J$ (if $F^{-1}(c)$ is not a singleton) or $J\setminus\{c\}$ (if $F^{-1}(c)$ is a singleton). Hence, since $I$ is connected, it is either $[a,b]$ or one of the half-open bounded intervals $[a,b)$ or $(a,b]$. Thus, if $I\subsetneq\R$ is open, then ${\rm r}(I)\geq 2$.

To finish the proof and in view of (i), it is enough to notice that the intervals $[0,1]$ and $(0,1]$ are regular images of $\R$ via the regular maps
$$
f_3:\R\to\R,\ t\mapsto\frac{t}{1+t^2}+\frac{1}{2},\qquad f_4:\R\to\R,\ t\mapsto\frac{1}{1+t^2}
$$
while the interval $(0,1)$ is a regular image of $\R^2$ via the regular map
$$
f_5:\R^2\to\R,\ (x,y)\mapsto\frac{(xy-1)^2+x^2}{1+(xy-1)^2+x^2}.
$$
The concrete details are left to the reader.
\end{proof}

\subsection*{Proof of Theorem \ref{dim1p}}
(i) $\Longrightarrow$ (ii) We know that $S$ is unbounded and, by \cite[3.1]{fg3}, $S$ is irreducible. Next, since ${\rm p}(S)<+\infty$, there exists a regular map $f:\R^n\to\R^m$ such that $f(\R^n)=S$. By Lemma \ref{fact2}, there are polynomial maps $h:\R\rightarrow\R^m$ and $g:\R^n\rightarrow\R$ satisfying $f=h\circ g$; notice that the Zariski closures of $f(\R^n)$ and $h(\R)$ coincide. Now, by Lemma \ref{fact1} applied to the polynomial map $h$, we conclude that $\cl_{\C\PP^m}^{\zar}(S)$ is an invariant rational curve such that $\cl_{\C\PP^m}^{\zar}(S)\cap\mathsf{H}_\infty(\C)=\{p\}$ is a singleton and the germ $\cl_{\C\PP^m}^{\zar}(S)_p$ is irreducible.

\vspace{1mm}
(ii) $\Longrightarrow$ (i) Let $\Pi:=(\Pi_0:\ldots:\Pi_m):\C\PP^1\to\cl_{\C\PP^m}^{\zar}(S)$ be an invariant normalization of $\cl_{\C\PP^m}^{\zar}(S)$; in particular, $\pi(\R\PP^1)=\cl_{\R\PP^m}^{\zar}(S)_{(1)}$. Since $\cl_{\C\PP^m}^{\zar}(S)\cap\mathsf{H}_\infty(\C)=\{p\}$ is a singleton and the germ $\cl_{\C\PP^m}^{\zar}(S)_p$ is irreducible, we may assume that
$$
\Pi^{-1}(\cl_{\C\PP^m}^{\zar}(S)\cap\mathsf{H}_\infty(\C))=\{(0:1)\};
$$
hence, $\Pi_0=\t_0^d$ for some $d\geq1$. Therefore, $\pi:=\Pi|_{\R}:\R\equiv\R\PP^1\setminus\{p_\infty\}\to\R^m$ is a polynomial map, and since $S$ is irreducible and $1$-dimensional, $S\subset\pi(\R)=\cl_{\R\PP^m}^{\zar}(S)_{(1)}\setminus\mathsf{H}_\infty(\R)$. Moreover, since $(\R,\pi)$ is the normalization of $\cl^{\zar}(S)$, there exists, by \cite[3.5]{fg3}, an interval $I\subset\R$ such that $\pi(I)=S$; in fact, since $S$ is unbounded, so is $I$. Now, by Lemma \ref{interval}, $I$ and so $S$ are polynomial images of $\R^2$, as wanted.
\qed

\subsection*{Proof of Proposition \ref{p12}}
If ${\rm p}(S)=1$, there exists a non-constant polynomial map $f:\R\to\R^m$ such that $f(\R)=S$. Now, since $f$ is proper, $S$ is closed in $\R^m$.

Conversely, as we have seen in the proof of (ii) $\Longrightarrow$ (i) in Theorem \ref{dim1p}, there exists a polynomial map $\pi:\R\to\R^m$ such that $(\R,\pi)$ is the normalization of $\cl^{\zar}(S)$. Thus, there exists, by \cite[3.5]{fg3}, an interval $I\subset\R$ such that $\pi(I)=S$. Since $S$ is unbounded and closed, such interval can also be taken unbounded and closed. Thus, by Lemma \ref{interval}, $I$ and so $S$ are polynomial images of $\R$.
\qed

\subsection*{Proof of Theorem \ref{dim1r}}
(i) $\Longrightarrow$ (ii) First, by \cite[3.1]{fg3}, $S$ is irreducible. Let now $f:\R^n\to\R^m$ be a regular map such that $f(\R^n)=S$. By Lemma \ref{fact2}, there exist a rational function $g\in\R(\x)$ and a rational map $h:=(\frac{h_1}{h_0},\ldots,\frac{h_m}{h_0}):\R\to\R^m$ such that $f=h\circ g$. Now, by Lemma \ref{fact1}, we deduce that $\cl_{\R\PP^m}^{\zar}(S)$ is a rational curve.

\vspace{1mm}
(ii) $\Longrightarrow$ (i) Let $\pi:\R\PP^1\to\cl_{\R\PP^m}^{\zar}(S)$ be the normalization of $\cl_{\R\PP^m}^{\zar}(S)$; recall that $\pi(\R\PP^1)=\cl_{\R\PP^m}^{\zar}(S)_{(1)}$. If $S=\cl_{\R\PP^m}^{\zar}(S)_{(1)}$, then $S$ is, by Example \ref{s1rp1}, a regular image of $\R$. On the other hand, if $S\neq\cl_{\R\PP^m}^{\zar}(S)_{(1)}$, we may assume that the image of the infinite point $p_\infty$ of $\R\PP^1$ under $\pi$ belongs to $\cl_{\R\PP^m}^{\zar}(S)_{(1)}\setminus S$. Now, by \cite[3.5]{fg3}, there exists an interval $I\subset\R=\R\PP^1\setminus\{p_\infty\}$ such that $\pi(I)=S$. By Lemma \ref{interval}, we conclude that $I$ and so $S$ are regular images of $\R^2$.
\qed

\subsection*{Proof of Proposition \ref{r12}}
Suppose first that ${\rm r}(S)=1$. Let $f:\R\to\R^m$ be a regular map such that $f(\R)=S\subset\cl^{\zar}(S)$. By Lemma \ref{fact1}, $f$ extends to a surjective regular map $F:\C\PP^1\to\cl_{\C\PP^m}^{\zar}(S)$ and we may decompose $F=\Pi\circ\widetilde{F}$ where $\widetilde{F}:\C\PP^1\to\C\PP^1$ is an invariant surjective regular map and $(\C\PP^1,\Pi)$ is an invariant normalization of $\C\PP^1$; we may assume that $p_\infty\in\Pi^{-1}(\mathsf{H}_\infty(\C))$. Since $f$ is a regular map, 
$$
\varnothing=F^{-1}(\mathsf{H}_\infty(\C))\cap\R=(\widetilde{F})^{-1}(\Pi^{-1}(\mathsf{H}_\infty(\C)))\cap\R.
$$
Thus, since $p_\infty\in\Pi^{-1}(\mathsf{H}_\infty(\R))$, we deduce that the image of $\widetilde{f}:=\widetilde{F}|_\R$ is contained in $\R$ and so $\widetilde{f}:\R\to\R$ is a regular map such that $f=\pi\circ\widetilde{f}$, where $\pi:=\Pi|_\R$. By Lemma \ref{interval}, we may assume $\widetilde{f}(\R)=\R$, $[0,\infty)$, $[0,1]$ or $[0,1)$. 

In case $\widetilde{f}(\R)=[0,1]$, then $\cl_{\R\PP^m}(S)=S$. Otherwise, let $q:=p_\infty$ in case $\widetilde{f}(\R)=\R$ or $[0,\infty)$, and $q:=1$ in case $\widetilde{f}(\R)=[0,1)$. Observe that $J:=\widetilde{f}(\R)\cup\{q\}$ is a closed subset of $\R\PP^1$; hence, its image $S\cup\{\pi(q)\}$ under $\pi$ is a closed subset of $\R\PP^m$ and so $\cl_{\R\PP^m}(S)=S\cup\{\pi(q)\}$. Thus, $\cl_{\R\PP^m}(S)\setminus S$ is either empty or a singleton.

Suppose now that $\cl_{\R\PP^m}(S)\setminus S=\{p:=\pi(q)\}$; hence, $\pi^{-1}(p)\cap\widetilde{f}(\R)=\varnothing$ because $S=f(\R)=\pi(\widetilde{f}(\R))$ and so $\pi^{-1}(p)\cap J=\{q\}$. Thus, $S_p=\pi((\widetilde{f}(\R))_q)$ and we conclude that the analytic closure of the germ $S_p$ is irreducible.

\vspace{1mm}
Conversely, by Theorem \ref{dim1r} and \cite[3.5]{fg3} there exists a connected subset $I\subset\R\PP^1$ such that $\pi(I)=S$, where $(\R\PP^1,\pi)$ is the normalization of $\cl_{\R\PP^m}^{\zar}(S)$. In fact, $I$ is the unique $1$-dimensional connected component of $\pi^{-1}(S)$. We distinguish two possibilities:

\vspace{1mm}
\noindent{\sc Case 1}. $\cl_{\R\PP^m}(S)=S$. Then, $S$ is closed in $\R\PP^m$ and so $I$ is either $\R\PP^1$ or a compact interval contained in $\R\PP^1$ that we may assume equal to $[0,1]$.

\vspace{1mm}
\noindent{\sc Case 2}. $\cl_{\R\PP^m}(S)\setminus S=\{p\}$ is a singleton and the analytic closure of the germ $S_p$ is irreducible. Observe that $\cl_{\R\PP^m}(S)=\pi(\cl_{\R\PP^1}(I))$ and, since the analytic closure of the germ $S_p$ is irreducible, we deduce that $(\pi|_{\cl_{\R\PP^1}(I)})^{-1}(p)=\{a\}$ is a singleton. Thus, $I=\cl_{\R\PP^1}(I)\setminus\{a\}$ and we may assume that either $I=[0,1)$ or $I=\R$.

\vspace{1mm}
In both cases we conclude, by Lemma \ref{interval} and Example \ref{s1rp1}, that $S$ is a regular image of $\R$, as wanted.
\qed

\end{document}